\documentclass[10pt,twocolumn,twoside]{IEEEtran}
\usepackage[comma,numbers,square,sort&compress]{natbib}
\usepackage{algorithmic}
\usepackage{bm}
\usepackage{amsmath,amssymb,latexsym, amsfonts}
\usepackage{amsthm}
\newtheorem{theorem}{Theorem}
\newtheorem{lemma}[theorem]{Lemma}

\newtheorem{definition}[theorem]{Definition}

\newtheorem{remark}[theorem]{Remark}

\usepackage{graphicx}
\usepackage{booktabs}
\usepackage{longtable}
\allowdisplaybreaks{}
\usepackage{tikz}
\usetikzlibrary{graphs}
\ifnum\pdfshellescape=1
\fi
\usetikzlibrary{positioning}

\newcommand{\I}{\mathbf{I}}
\newcommand{\0}{\mathbf{0}}

\title{A Data-Driven Convex Programming Approach to Worst-Case Robust Tracking Controller Design}

\author{Liang Xu*, Mustafa Sahin Turan*, Baiwei Guo, Giancarlo Ferrari-Trecate
\thanks{Authors are with the Institute of Mechanical Engineering (IGM),
  EPFL, Switzerland. Email: {\tt\small  \{liang.xu, mustafa.turan, baiwei.guo, giancarlo.ferraritrecate\}@epfl.ch}}
    
 \thanks{*The first two authors contributed equally to this work. This work has been supported by the Swiss National Science Foundation under the COFLEX project (grant number 200021\_169906) and the National Centre of Competence in Research (NCCR) in Dependable and Ubiquitous Automation.}}

\begin{document}

\maketitle 
\begin{abstract}
  This paper studies finite-horizon robust tracking control for discrete-time linear systems, based on input-output data.
  We leverage behavioral theory to represent system trajectories through a set of noiseless historical data, instead of using an explicit system model.
  By assuming that recent output data available to the controller  are affected by noise terms verifying a quadratic bound, we formulate an optimization problem with a linear cost and LMI constraints for solving the robust tracking problem without any approximations.
  Our approach hinges on a parameterization of noise trajectories compatible with the data-dependent system representation and on a reformulation of the tracking cost, which enables the application of the S-lemma.
  In addition, we propose a method for reducing the computational complexity and demonstrate that the size of the resulting LMIs does not scale with the number of historical data.
 Finally, we show that the proposed formulation can easily incorporate actuator disturbances as well as constraints on inputs and outputs.
  The performance of the new controllers is discussed through simulations.
\end{abstract}

\section{Introduction}\label{sec.Introduction}

        Due to the recent advances in sensing, communication, and computation, data availability for control design is steadily increasing.
        This has motivated a renewed interest in  system analysis and control design methods relying on finite-length data sequences~\cite{DePersis2020TAC, vanWaardeHenkJ2020TAC, MatniNikolai2019CDCLearningControlSurvey, StephenTu2019Thesis}.
        Several recent works propose to use raw measurements for representing discrete-time systems, and solving system analysis and control design problems~\cite{Mishra2020DataDrivenControllabilityTest,Koch2020DataVerifyDispassivity,Romer2019DataVerifyPassivity,DePersis2020TAC, vanWaardeHenkJ2020TAC, vanWaardeHenk2020ArXiv, HarryTrentelman2020ArXiv, berberich2020robust, BisoffiAndrea2020RobustInvariance, vanWaardeHenk2020MatrixSLemma, DePersis2020LowComplexity, CoulsonJeremy2019ECCDeePC, BerberichJulian2020DeeMPCStability, CoulsonJeremy2020DistributionallyRobustDeePC}.
      As mentioned in~\cite{DePersis2020TAC}, the main feature of these approaches is to bypass explicit system identification that is usually required in standard control design.
        Moreover, data-based system representations can be easier to update when new data are available \cite{alpago2020extended}, hence facilitating the deployment of adaptive control systems. 
        
        All above works assume the availability of \textit{historical data}, i.e., finite-length trajectories produced by the open-loop system and measured offline.
        The works~\cite{DePersis2020TAC, vanWaardeHenkJ2020TAC, vanWaardeHenk2020ArXiv, HarryTrentelman2020ArXiv, berberich2020robust, BisoffiAndrea2020RobustInvariance, vanWaardeHenk2020MatrixSLemma, DePersis2020LowComplexity}  consider system representations based on input-state historical data.
Data-based parameterizations of linear state-feedback and linear quadratic regulators are developed in~\cite{DePersis2020TAC}, under the assumption that historical input data are persistently exciting.
This assumption further implies that a state-space model of the system can be perfectly reconstructed for control design.
The persistence of excitation requirement is relaxed in~\cite{vanWaardeHenkJ2020TAC}, where the authors provide necessary and sufficient conditions about the informativity of historical data for testing system properties and building stabilizing control laws.
The informativity framework is further extended to sub-optimal control design in~\cite{vanWaardeHenk2020ArXiv}, as well as tracking and regulation problems in~\cite{HarryTrentelman2020ArXiv}.
The presence of noise in historical data, which prevents from unambiguously identifying the system dynamics, is considered in~\cite{DePersis2020LowComplexity, berberich2020robust, BisoffiAndrea2020RobustInvariance, vanWaardeHenk2020MatrixSLemma}.
\cite{DePersis2020LowComplexity} extends the optimal control method in \cite{DePersis2020TAC} to account for noisy data and derives sufficient conditions to ensure that the proposed method returns a stabilizing controller.
In~\cite{berberich2020robust, BisoffiAndrea2020RobustInvariance, vanWaardeHenk2020MatrixSLemma}, data are first used for representing all systems that are compatible with available prior knowledge on the noise and then for developing different kinds of state-feedback regulators, including robustly stabilizing, $\mathcal{H}_2$, and $\mathcal{H}_{\infty}$ controllers.

In certain applications, the system state is not accessible and only input-output data can be collected.
In this scenario, Willems' Fundamental Lemma states that the whole set of input-output trajectories generated by a discrete-time linear system can be represented by finitely many historical data coming from sufficiently excited dynamics~\cite{WillemsJan2005SCL}.
In view of this result,~\cite{MarkovskyIvan2008DataDrivSimCont} proposes to predict the system output from a given time $t$ onwards by using a set of collected historical data and a finite amount of \textit{recent past data}, i.e., an input-output trajectory measured right before time $t$. This approach is also used in the data-enabled predictive control (DeePC) scheme described in~\cite{CoulsonJeremy2019ECCDeePC}.
While originally developed for noiseless data, DeePC has been recently extended to noisy trajectories in~\cite{BerberichJulian2020DeeMPCStability, CoulsonJeremy2020DistributionallyRobustDeePC}.
In~\cite{BerberichJulian2020DeeMPCStability}, slack variables are introduced in the data-dependent system representation to account for  noisy measurements.
The modified control scheme is shown to be recursively feasible and practically exponentially stable; however, the tracking performance is not analyzed.
The authors in~\cite{CoulsonJeremy2020DistributionallyRobustDeePC} propose a distributionally robust variant of DeePC based on semi-infinite optimization.
They then formulate a finite and convex program, whose optimal value is an upper bound to that of the original optimization problem.
The work~\cite{RN11435} considers using noiseless historical data and noisy recent output data to minimize the energy of the control input while robustly satisfying input/output constraints.
The authors propose to separate the problems of estimation of the initial condition and control design, and show that the solution to the formulated problem is computed by consecutively solving two optimization problems.

In safety-critical applications, such as power networks and industrial control systems, it is sometimes required to adopt a bounded-error perspective by enforcing robustness against all possible noise realizations and providing worst-case performance guarantees.
This is the setting considered in the present paper and, for this purpose, we utilize the data-driven prediction method in~\cite{MarkovskyIvan2008DataDrivSimCont}.
We assume the historical data are noiseless while recent data are corrupted by noise terms satisfying a quadratic constraint similar to the one in~\cite{vanWaardeHenk2020MatrixSLemma}.
This assumption corresponds to scenarios where one can utilize very accurate (and, thus, expensive) sensors to collect offline historical data, but only has relatively inaccurate and noisy sensors to be used during online operations.
Our goal is to provide a control design method for worst-case optimal reference tracking with explicit performance guarantees.

We first characterize noises that are consistent with the input-output data, and then reformulate the tracking cost.
This enables us to apply the S-lemma~\cite{polik2007survey} to transform the worst-case robust control problem  to an equivalent minimization problem with a linear cost and LMI constraints.
Moreover, we propose a method for reducing the size of LMI constraints, and also show that our formulation can easily incorporate input-output constraints as well as actuator disturbances.
In contrast to~\cite{RN11435}, we consider to minimize quadratic cost on both inputs and outputs, while the method in~\cite{RN11435} only deals with the minimization of the input energy.
The main features of our method are the following: (1) we consider the minimization of the worst-case tracking performance; (2) the proposed design procedure is non-conservative, meaning that we obtain the optimal tracking controllers without any approximations; (3) the complexity of the controller design procedure does not increase with the number of historical data.
To the authors' knowledge, this is the first time that the worst-case robust optimal tracking control is considered and exactly solved in a data-driven fashion.

A preliminary version of this work has been submitted to a conference~\cite{RN11384}.
With respect to it, this paper provides complete proofs for all intermediate results, contains new results on reducing the  LMI constraint size, illustrates how to consider input-output constraints as well as actuator disturbances, and adds new numerical experiments.
This paper is organized as follows. In Section~\ref{sec.PreliminariesDataDriven}, we provide preliminaries on data-driven prediction.
The problem formulation is given in Section~\ref{sec.ProblemFormulation}.
The data-based robust optimal tracking control problem is solved in Section~\ref{sec.RobustControlDesign}.
Extensions for considering input-output constraints as well as actuator disturbances are discussed in Section~\ref{sec:Extensions}.
Simulations are provided in Section~\ref{subsec.Simulation}.
Concluding remarks are given in Section~\ref{sec.Conclusion}.

\textbf{Notation}:
$\mathrm{col}(\{x_k\}_{k=i}^j)$ denotes the column concatenation of the vectors $x_k$ in the sequence $\{x_k\}_{k=i}^j$. 
For a square matrix $\Phi$, $\Phi> 0$ $(\Phi\ge 0)$ represents that it is positive definite (semidefinite).
For $Q\geq0$, $\|x\|_Q$ denotes $\sqrt{x^\top Qx}$.
$\mathrm{dim}(V)$ denotes the dimension of the vector space $V$.
For a matrix $A \in \mathbb{R}^{n\times m}$, $\ker(A)$ and $\mathrm{range}(A)$ denote its null space and column space, respectively. 
Moreover, $\mathcal{N}(A)\in \mathbb{R}^{m\times \dim(\ker(A))}$ and $\mathcal{R}(A)\in \mathbb{R}^{n\times \mathrm{rank}(A)}$ denote matrices whose columns form a basis for $\mathrm{ker}(A), \mathrm{range}(A)$, respectively.
$\I$ and $\0$ denote identity and zero matrices of suitable size.
When used with subspaces, the operator $+$ denotes the subspace sum.
The operator $\otimes$ denotes the Kronecker product. 

\section{Preliminaries on Data-Driven Prediction}\label{sec.PreliminariesDataDriven}

We consider a controllable and observable discrete-time LTI system $\mathcal{G}$ with state-space representation
\begin{equation}\label{eq.LTI_system_ss} 
\begin{aligned}
x_{k+1} &=A x_{k}+B u_{k}, \\
y_{k} &=C x_{k}+D u_{k},
\end{aligned}
\end{equation}
where $x_k\in \mathbb{R}^n$, $u_k\in \mathbb{R}^m$, $y_k\in \mathbb{R}^p$ are the  system state, input and output, respectively.
In this paper, we assume that system matrices $(A, B, C, D)$ are unknown, the states $x_k$ are not measurable, and only a finite set of input-output samples of $\mathcal{G}$ is available.
In this section, we recall how to form a data-based representation of $\mathcal{G}$ that allows for predicting the output given any input~\cite{WillemsJan2005SCL, MarkovskyIvan2008DataDrivSimCont}.

We start by introducing the following definitions.
The \textit{lag} $\mathbf{l}(\mathcal{G})$ is the smallest integer $l$ such that the $l$-step observability matrix $[C^\top,A^\top C^\top,\dots,(A^{l-1})^\top C^\top]^\top$ has full column rank.
Moreover, $\mathbf{l}(\mathcal{G})\leq n$ since $\mathcal{G}$ is observable.
A sequence ${\{u_{k}, y_{k}\}}_{k=l}^{l+T-1}$ is a \textit{trajectory} of $\mathcal{G}$ if and only if there exists a state sequence ${\{x_{k}\}}_{k=l}^{l+T}$ such that~\eqref{eq.LTI_system_ss} holds for $k=l, \ldots, l+T-1$.
For a sequence ${\{v_k\}}_{k=i}^j$ of vectors, we use $v$ to denote $\mathrm{col}({\{v_k\}}_{k=i}^j)$ when the start and end times $i, j$ are clear from the context.
The \textit{Hankel matrix of depth} $L$ corresponding to a sequence ${\{v_k\}}_{k=l}^{l+T-1}$ is defined as
\begin{align*}
\mathcal{H}_{L}(v):=\left[\begin{array}{cccc}v_{l} & v_{l+1} & \cdots & v_{l+T-L} \\ v_{l+1} & v_{l+2} & \cdots & v_{l+T-L+1} \\ \vdots & \vdots & \ddots & \vdots \\ v_{l+L-1} & v_{l+L} & \cdots & v_{l+T-1}\end{array}\right].
\end{align*}
The sequence ${\{v_{k}\}}_{k=l}^{l+T-1}$ is \textit{persistently exciting of order} $L$ if the  Hankel matrix $\mathcal{H}_{L}(v)$ is of full row rank.

In the following, we introduce the input-output data-based representation of linear systems in~\cite{WillemsJan2005SCL} and the  prediction method in~\cite{MarkovskyIvan2008DataDrivSimCont}.
Suppose a trajectory ${\{\bar{u}_{k}, \bar{y}_{k}\}}_{k=t_h}^{t_h+T_d-1}$ of $\mathcal{G}$ is collected, where  $t_h\ll 0$.
The trajectory is called \textit{historical}, since it can be regarded as collected long before the start (indicated by time $0$) of any control or prediction tasks.
The Fundamental Lemma proposed by Willems et al.~\cite{WillemsJan2005SCL} shows how to use the historical trajectory  to characterize all possible system trajectories of length $T_f$.
 \begin{lemma}[Fundamental Lemma~\cite{WillemsJan2005SCL}]\label{lem.WillemFundamentalLemma}
   Suppose that ${\{\bar{u}_{k}, \bar{y}_{k}\}}_{k=t_h}^{t_h+T_d-1}$ is a trajectory of $\mathcal{G}$ and that the input $\bar{u}$ is persistently exciting of order $T_f+n$.
   Then, ${\{u_{k}, y_{k}\}}_{k=0}^{T_f-1}$ is a trajectory of $\mathcal{G}$ if and only if there exists $g \in \mathbb{R}^{T_d-T_f+1}$ such that
   \begin{align}\label{eq.WillemsLemma}
     \left[\begin{array}{c}
             \mathcal{H}_{T_f}(\bar{u}) \\
             \mathcal{H}_{T_f}(\bar{y})
           \end{array}\right]g =\left[\begin{array}{l}
                                        u \\
                                        y
                                      \end{array}\right].
   \end{align}
\end{lemma}

For a given time $t\ge 0$, consider the problem of using~\eqref{eq.WillemsLemma} for computing predictions $y=\mathrm{col}(\{y_k\}_{k=0}^{T_f-1})$ of the system output over a future horizon given the inputs $u=\mathrm{col}(\{u_k\}_{k=0}^{T_f-1})$. There are infinitely many output trajectories $y$ that satisfy~\eqref{eq.WillemsLemma}, corresponding to different initial states $x_0$.
The authors of~\cite{MarkovskyIvan2008DataDrivSimCont} propose to implicitly fix the initial state by using recent input-output samples\footnote{We call $\{u_{\mathrm{ini}}, y_{\mathrm{ini}}\}$ \textit{recent data}.} $u_{\mathrm{ini}}=\mathrm{col}(\{u_k\}_{k=-T_{\mathrm{ini}}}^{-1})$, $y_{\mathrm{ini}}=\mathrm{col}(\{y_k\}_{k=-T_{\mathrm{ini}}}^{-1})$, which are available at time $0$ (see Fig.~\ref{fig:data_driven_prediction}).
More precisely, let 
\begin{equation*}
	U = \begin{bmatrix}
	U_p \\
	U_f
	\end{bmatrix} \triangleq \mathcal{H}_{T_{\mathrm{ini}}+T_f}\left(\bar{u}\right), \quad Y=\begin{bmatrix}
	Y_p \\
	Y_f
	\end{bmatrix}\triangleq \mathcal{H}_{T_{\mathrm{ini}}+T_f}\left(\bar{y}\right),
\end{equation*}
where $U_p$ and $Y_p$ consist of the first $T_\mathrm{ini}$ block rows of $U$ and $Y$, respectively; while $U_f$ and $Y_f$ consist of the last $T_f$ block rows of $U$ and $Y$, respectively.
The following lemma summarizes the prediction algorithm.
\begin{lemma}[\cite{MarkovskyIvan2008DataDrivSimCont}]\label{lemma.UniqueOutput}
Suppose that $\bar{u}$ is persistently exciting of order $T_{\mathrm {ini}}+T_f+n$, and  $T_{\mathrm {ini}} \geq \mathbf{l}(\mathcal{G})$.
Then, for a given recent system trajectory $(u_\mathrm{ini},y_\mathrm{ini})$ and the input sequence $u$,
\begin{enumerate}
\item there exists at least one vector $g$ verifying
  \begin{align} \label{eq.DataDriveSimControl1}
  \left[\begin{array}{l}
          U_p \\
          Y_p \\
          U_f 
        \end{array}\right] g=\left[\begin{array}{c}
                                     u_{\mathrm{ini}} \\
                                     y_{\mathrm {ini}} \\
                                     u
                                   \end{array}\right],
  \end{align}
\item  the output prediction $y$ is unique and given by
  \begin{align} \label{eq.DataDriveSimControl2}
y=Y_fg,
  \end{align}
  for any $g$ satisfying~\eqref{eq.DataDriveSimControl1}.
\end{enumerate}
\end{lemma}

Note that, collectively~\eqref{eq.DataDriveSimControl1} and~\eqref{eq.DataDriveSimControl2} are equivalent to~\eqref{eq.WillemsLemma}, which can be rewritten as
\begin{align*}
  \left[\begin{array}{l}
          U_p \\
          Y_p \\
          U_f \\
          Y_f
        \end{array}\right] g=\left[\begin{array}{c}
                                     u_{\mathrm{ini}} \\
                                     y_{\mathrm {ini}} \\
                                     u\\
                                     y
                                   \end{array}\right].
\end{align*}

  Throughout the paper, we assume that $\bar{u}$ and $T_\mathrm{ini}$ verify the conditions in Lemma~\ref{lemma.UniqueOutput}, which implies that the matrices $U$, $U_p$ and $U_f$ have full row rank.
\begin{figure}[t]
    \centering
    \includegraphics[width = 8cm]{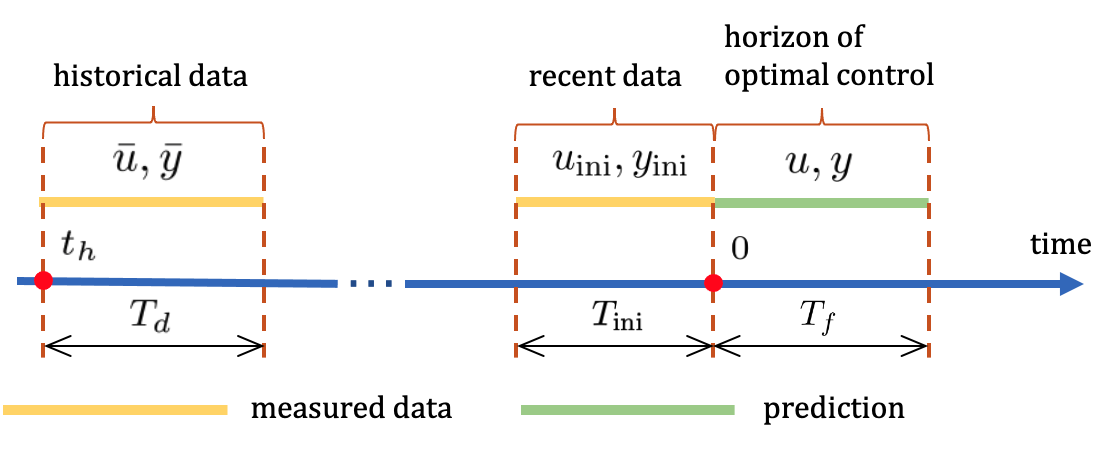}
    \caption{Chronological order of data used in data-driven prediction}
    \label{fig:data_driven_prediction}
\end{figure}  

\section{Problem Formulation}\label{sec.ProblemFormulation}

In view of the prediction algorithm described in Lemma~\ref{lemma.UniqueOutput}, consider the following data-driven linear-quadratic tracking problem
\begin{gather}
\min_{u, y, g}\sum_{k=0}^{T_f-1}\left(\left\|y_{k}-r_{k}\right\|_{Q}^{2}+\left\|u_{k}\right\|_{R}^{2}\right) \label{eq.yini-performanceconstraint-noiseless}~\text{s.t.} ~ \eqref{eq.DataDriveSimControl1},\eqref{eq.DataDriveSimControl2}
\end{gather}
where $\{r_k\}_{k=0}^{T_f-1}$ is the tracking reference and $Q\geq 0$, $R>0$ are weight matrices.
In this paper, we assume that the noisy recent outputs $y_{\mathrm{ini}}$ verify
\begin{align*}
y_{\mathrm{ini}}=\check{y}_{\mathrm{ini}}+w,
\end{align*}
where $\check{y}_{\mathrm{ini}}$ represents the noiseless output and $w$ denotes the measurement noise.
Besides, as in~\cite{vanWaardeHenk2020MatrixSLemma} and~\cite{berberich2020robust}, we assume that $w$ satisfies the following quadratic constraint 
\begin{align}
  \begin{bmatrix}
    1\\
    w
  \end{bmatrix}^\top
  \underbrace{\begin{bmatrix}
    \Phi_{11} & \Phi_{12}\\
    \Phi_{12}^\top & \Phi_{22}
  \end{bmatrix}}_{\Phi}
                 \begin{bmatrix}
                   1\\
                   w
                 \end{bmatrix}
  \ge 0, \label{eq.yini-noiseconstraint}
\end{align}
where $\Phi_{22}=\Phi_{22}^\top< 0$.

\begin{remark}\label{rem.wbounded}
The negative definiteness of $\Phi_{22}$ ensures that $\|w\|_2$ is bounded.
  In the special case that $\Phi_{12}=\0$ and $\Phi_{22}=-\I$,~\eqref{eq.yini-noiseconstraint} reduces to
  \begin{align*}
    w^\top w=\sum_{k}w_k^\top w_k\le \Phi_{11},
  \end{align*}
  which, as highlighted in~\cite{vanWaardeHenk2020MatrixSLemma}, has the interpretation of bounded noise energy.\texttt{}
\end{remark}
We are interested in designing a control input $u$ that minimizes the worst-case quadratic tracking error among all feasible noise trajectories, which are defined as follows.
\begin{definition}
	For recent data $(u_{\mathrm{ini}}, y_{\mathrm{ini}})$, a noise trajectory $w$ is called feasible if it verifies~\eqref{eq.yini-noiseconstraint} and $(u_{\mathrm{ini}}, y_{\mathrm{ini}}-w)$ is a trajectory of $\mathcal{G}$.
\end{definition}
Next, we provide a robust formulation of the tracking problem~\eqref{eq.yini-performanceconstraint-noiseless} based on the linear quadratic tracking error
$$\text{ LQTE}(u,y,w)\triangleq \sum_{k=0}^{T_f-1}\left(\left\|y_{k}-r_{k}\right\|_{Q}^{2}+\left\|u_{k}\right\|_{R}^{2}\right).$$

\noindent \textbf{Problem P1}: Find the input sequence $u$ solving
\begin{subequations}
  \label{eq.optP1}
  \begin{align}
  \min_{u, y, g} \max_{w} & \quad \text{ LQTE}(u,y,w)\label{eq.yini-performanceobjective}\\
    \text{subject to} &\quad 
      \begin{bmatrix}
        U_p \\
        Y_p \\
        U_f \\
        Y_f
      \end{bmatrix} g=\begin{bmatrix}
        u_{\mathrm{ini}} \\
        y_{\mathrm{ini}} \\
        u \\
        y
      \end{bmatrix}-\begin{bmatrix}
        \0 \\
        w \\
        \0 \\
        \0
      \end{bmatrix}, \label{eq.yini-model-minmax}\\
&\quad       w \text{ is a feasible noise trajectory} \label{eq.yini-noiseconstraint-minmax}.
\end{align}
\end{subequations}

The constraint~\eqref{eq.yini-noiseconstraint-minmax} makes the min-max optimization problem difficult to solve.
However, as we show in the next section, this issue can be circumvented by using a suitable parameterization of feasible noise trajectories.

\section{Robust Controller Design}\label{sec.RobustControlDesign}
Problem \textbf{P1} can be reformulated as follows
\begin{equation}\label{eq.minimization_problem}
  \begin{gathered}
    \min_{u, \gamma, g, y} \quad \gamma \\
    \text{s.t.}, \text{LQTE}(u,y,w)\leq \gamma,~ \forall w
    \text{ satisfying }~\eqref{eq.yini-noiseconstraint-minmax}.
  \end{gathered}
\end{equation}
For notational simplicity, we have omitted the dependence of the problem on $r$.
In the sequel, we will derive a tractable reformulation of~\eqref{eq.minimization_problem}.
We first show in Section~\ref{subsec.NoiseConstTransform} that any noise trajectory $w$ fulfilling~\eqref{eq.yini-noiseconstraint-minmax} can be expressed as an affine function of a vector $g_w$ satisfying a quadratic constraint.
In Section~\ref{subsec.PerformanceTransformation}, we show that the output $y$ is completely determined by the input $u$ and the vector $g_w$, which allows one to express the tracking error constraint as a quadratic constraint on $g_w$.
In light of these results, in Section~\ref{subsec.MainResult}, we prove that~\eqref{eq.minimization_problem} is equivalent to a minimization problem with a linear cost and LMI constraints.
Finally, in Section~\ref{subsec.ReducingDimension}, we show how to reduce the size of the LMI constraints to reduce the computational burden.

\subsection{Feasible Noise Parameterization}\label{subsec.NoiseConstTransform}
Since $\bar{u}$ is persistently exciting of order $T_{\mathrm{ini}}+T_f+n$, it is also persistently exciting of order $T_{\mathrm{ini}}+n$. In view of Lemma~\ref{lem.WillemFundamentalLemma},  $(u_\mathrm{ini},y_\mathrm{ini}-w)$ is a trajectory of $\mathcal{G}$ if and only if there exists $g_\mathrm{ini} \in \mathbb{R}^{T_d-T_\mathrm{ini}-T_f+1}$ such that
\begin{align}
\begin{bmatrix}
u_{\mathrm{ini}}\\
y_{\mathrm{ini}}-w
\end{bmatrix}=\begin{bmatrix}
U_p\\
Y_p
\end{bmatrix}g_\mathrm{ini}. \label{eq.yini-trajectoryconstraint}
\end{align}
Consider the solution $	g_{\mathrm{ini}}^* = U_p^\top \left(U_pU_p^\top\right)^{-1}u_{\mathrm{ini}}$  to the first equation in~\eqref{eq.yini-trajectoryconstraint}, i.e., $U_pg_{\mathrm{ini}}^*=u_{\mathrm{ini}}$. Any other solution $g_\mathrm{ini}$ verifying $U_pg_\mathrm{ini}=u_{\mathrm{ini}}$ can be written as $g_\mathrm{ini}=g_{\mathrm{ini}}^* +  Mg_w$ for some  $g_w\in \mathbb{R}^{T_d-(m+1)T_{\mathrm{ini}}-T_f+1}$, where	$M = \mathcal{N}(U_p)$.
Furthermore, from the second block row of~\eqref{eq.yini-trajectoryconstraint}, any $w$ that makes $(u_\mathrm{ini},y_\mathrm{ini}-w)$ a trajectory of $\mathcal{G}$ can be written as
	\begin{equation}\label{eq.w_parameterization}
		w = -Y_pMg_w + \underbrace{(- Y_pg_{\mathrm{ini}}^*+y_\mathrm{ini})}_{w_0}
	\end{equation}
for some $g_w$.

In view of the above results, the feasible noise trajectories can be explicitly parameterized as follows.
\begin{lemma}\label{lem.g_w_parameterization}
  The noise trajectory $w$ is feasible if and only if there exists $g_w$ satisfying~\eqref{eq.w_parameterization} and
	\begin{equation}\label{eq.g_w_constraint}
		\begin{bmatrix}
		1\\
		g_w
		\end{bmatrix}^\top
		\underbrace{\begin{bmatrix}
			{[A_w]}_{11} & {[A_w]}_{12} \\
			{[A_w]}_{12}^\top & {[A_w]}_{22}
			\end{bmatrix}}_{A_w}
		\begin{bmatrix}
		1\\
		g_w
		\end{bmatrix}\geq 0,
	\end{equation}
	where
	\begin{equation}\label{eq.A_w}
		\begin{split}
		{[A_w]}_{11}&=  \Phi_{11}+w_0^\top\Phi_{12}^\top+\Phi_{12}w_0+w_0^\top\Phi_{22}w_0, \\
		{[A_w]}_{12}&= -\Phi_{12}Y_pM-w_0^\top\Phi_{22}Y_pM, \\
		{[A_w]}_{22}&= M^\top Y_p^\top\Phi_{22}Y_pM.
		\end{split}
	\end{equation}
\end{lemma}

\begin{proof}
  In addition to making $(u_{\mathrm{ini}}, y_{\mathrm{ini}}-w)$ a trajectory of $\mathcal{G}$, a feasible $w$ should also satisfy the constraint~\eqref{eq.yini-noiseconstraint}.
  Substituting~\eqref{eq.w_parameterization} into~\eqref{eq.yini-noiseconstraint}, we write the constraint on $g_w$ as
\begin{align*}
 \begin{bmatrix}
    1\\
    w
  \end{bmatrix}^\top\Phi
  \begin{bmatrix}
    1\\
    w
  \end{bmatrix}
  &= \Phi_{11}+(-Y_pMg_w+w_0)^\top\Phi_{12}^\top \\
  &+\Phi_{12}(-Y_pMg_w+w_0)\\
  &+(-Y_pMg_w+w_0)^\top\Phi_{22}(-Y_pMg_w+w_0)\\
  &=
    \begin{bmatrix}
      1\\
      g_w
    \end{bmatrix}^\top
{A_w}
                                     \begin{bmatrix}
                                       1\\
                                       g_w
                                     \end{bmatrix},
\end{align*}
with $A_w$ defined in~\eqref{eq.g_w_constraint}, \eqref{eq.A_w}.
\end{proof}

\begin{remark}
  We note that, from~\eqref{eq.w_parameterization}, for a given vector $g_{\mathrm{ini}}^*$, there might be multiple $g_w$ parameterizing the same noise trajectory $w$.
  In Section~\ref{subsec.ReducingDimension}, we show that this redundancy increases the computational complexity of the proposed method, and provide a method to overcome this problem.
\end{remark}

\subsection{Reformulation of the Tracking Error Constraint}\label{subsec.PerformanceTransformation}

In this section, we show that for a feasible noise trajectory $w$, the tracking error constraint LQTE$(u,y,w)\leq\gamma$ can be reformulated as a quadratic constraint on the parameter vector $g_w$.
We achieve this goal by first writing the output $y$ as an affine function of $u$ and $g_w$, and then substituting the expression of $y$ into the tracking error constraint.

To express $y$ in terms of $u$ and $g_w$, we compute $g$ from the first three block rows of~\eqref{eq.yini-model-minmax}, and substitute it into the last block row of~\eqref{eq.yini-model-minmax}.
First of all, we show how to construct a solution $g$ from 
\begin{align}
  \begin{bmatrix}
    U_p\\
    Y_p\\
    U_f
  \end{bmatrix}g=
  \begin{bmatrix}
    u_{\mathrm{ini}}\\
    y_{\mathrm{ini}}-w\\
    u
  \end{bmatrix}.\label{eq.yini-FirstThreeSolution}
\end{align}
Since $(u_{\mathrm{ini}}, y_{\mathrm{ini}}-w)$ is a feasible system trajectory, in view of Lemma~\ref{lemma.UniqueOutput}, for any given input $u$, there exists a (possibly nonunique) vector $g$ verifying~\eqref{eq.yini-FirstThreeSolution}.
Any solution $g$ to~\eqref{eq.yini-FirstThreeSolution}  can be decomposed as $g\triangleq g_\mathrm{ini} + g_u$, where $g_\mathrm{ini}$ verifies~\eqref{eq.yini-trajectoryconstraint} and $g_u$ solves

\begin{equation}\label{eq.g_u}
	\begin{bmatrix}
	U_p\\
	Y_p\\
	U_f
	\end{bmatrix}g_u=
	\begin{bmatrix}
	\mathbf{0}\\
	\mathbf{0}\\
	u-U_fg_\mathrm{ini}
	\end{bmatrix}.
\end{equation}
Therefore, if we can find a solution $g_u$ to~\eqref{eq.g_u}, we can obtain a solution $g$ to~\eqref{eq.yini-FirstThreeSolution}.

Before showing how to construct $g_u$ in Lemma~\ref{lemma.Yp2LinearCombination}, the following results are needed.
In view of Theorem 2 of~\cite{RN11380},  the matrix $[U_p^\top,Y_p^\top,U_f^\top]^\top$ does not always have full row rank, even though $[U_p^\top,U_f^\top]^\top$ does.
Therefore, there exists a row permutation matrix $P_Y$ transforming $Y_p$ as $P_YY_p = [Y_{p1}^\top,Y_{p2}^\top]^\top$
 such that $\Lambda \triangleq[U_p^\top,Y_{p1}^\top,U_f^\top]^\top$ has full row rank and
 \begin{align*}
\text{rank}\left(\Lambda\right)=\text{rank}\left(\begin{bmatrix}
 U_p \\
 Y_{p} \\
 U_f
 \end{bmatrix}\right).   
 \end{align*}
 
\begin{lemma}\label{lemma.Yp2LinearCombination}
A solution to~\eqref{eq.g_u} is given by
	\begin{equation}\label{eq.g_u_solution}
		g_u=\Lambda^\top (\Lambda\Lambda^\top)^{-1}\begin{bmatrix}
		\mathbf{0}\\
		\mathbf{0}\\
		u-U_fg_\mathrm{ini}
		\end{bmatrix}.
	\end{equation}  
\end{lemma}

\begin{proof}
  Left multiply both sides of~\eqref{eq.g_u} with  $\left[ \begin{smallmatrix}
  \mathbf{I} & \mathbf{0} & \mathbf{0} \\
  \mathbf{0} & P_Y & \mathbf{0} \\
  \mathbf{0} & \mathbf{0} & \mathbf{I}
\end{smallmatrix}\right]$ to obtain
\begin{align}
\begin{bmatrix}
  U_p\\
  Y_{p1}\\
  Y_{p2}\\
  U_f
  \end{bmatrix}g_u=
  \begin{bmatrix}
    \0\\
    \0\\
    \0\\
    u-U_fg_\mathrm{ini}
  \end{bmatrix}.
  \label{eq.yini-initialtrajectoryconstraint}
\end{align}
By definition, the rows of $Y_{p2}$ can be written as linear combinations of the rows of $[U_p^\top,Y_{p1}^\top,U_f^\top]^\top$. Therefore there exists an ordered sequence of elementary row operations $\{E_k\}_{k=1}^e$ captured by the matrix $E\triangleq E_eE_{e-1}\dots E_1$ such that
\begin{align}
E\begin{bmatrix}
U_p\\
Y_{p1}\\
Y_{p2}\\
U_f
\end{bmatrix}=
\begin{bmatrix}
U_p\\
Y_{p1}\\
\0\\
U_f
\end{bmatrix}. \label{eq.Eproperty1}
\end{align}
We next show by contradiction that the rows of $Y_{p2}$ can also be written as linear combinations of the rows of $[  U_p^\top,  Y_{p1}^\top]^\top$. 
Suppose that the rows of $Y_{p2}$ cannot be written as linear combinations of the rows of $[U_p^\top,Y_{p1}^\top]^\top$.
Then, left multiplying $E$ to both sides of~\eqref{eq.yini-initialtrajectoryconstraint}, we obtain
\begin{align}
  \begin{bmatrix}
    U_p\\
    Y_{p1}\\
    \0\\
    U_f
  \end{bmatrix}g_u=
  \begin{bmatrix}
    \0\\
    \0\\
    \text{linear combination of rows of }u-U_fg_\mathrm{ini}\\
    u-U_fg_\mathrm{ini}
  \end{bmatrix}. \label{eq: permutationcontradiction}
\end{align}
Since $[\0^\top,\0^\top]^\top$ is a feasible system trajectory, in view of Lemma~\ref{lemma.UniqueOutput},  there always exists a $g_u$ solving~\eqref{eq.g_u}.
Therefore,~\eqref{eq.g_u} and further~\eqref{eq.yini-initialtrajectoryconstraint},~\eqref{eq: permutationcontradiction} should always be compatible for any $u-U_fg_{\mathrm{ini}}$.
However, it is clear from the third block row of~\eqref{eq: permutationcontradiction} that,~\eqref{eq: permutationcontradiction} is not always compatible for any  $u-U_fg_{\mathrm{ini}}$.
This makes a contradiction.
Therefore,  the rows of $Y_{p2}$ can be written as linear combinations of the rows of $[  U_p^\top,  Y_{p1}^\top]^\top$. 
As a result, the matrix $E$ can be constructed such that
\begin{align}
  	E
	\begin{bmatrix}
	\0\\
	\0\\
	\0\\
	u-U_f g_\mathrm{ini}
	\end{bmatrix}=
	\begin{bmatrix}
	\0\\
	\0\\
	\0\\
	u-U_f g_\mathrm{ini}
	\end{bmatrix},\label{eq.Eproperty2}
\end{align}
that is, the matrix $E$ only applies elementary row operations on the first three block rows of $	[
	\0^\top,
	\0^\top,
	\0^\top,
	(u-U_f g_\mathrm{ini})^\top]^\top.$

The vector $g_u$ in~\eqref{eq.g_u_solution} satisfies
	  \begin{align*}
	\begin{bmatrix}
	U_p\\
	Y_{p1}\\
	U_f
	\end{bmatrix}g_u=
	\begin{bmatrix}
	\0\\
	\0\\
	u-U_f g_\mathrm{ini}
	\end{bmatrix}.
	\end{align*}
	Then, we have
	\begin{align} \label{eq.temp1}
	\begin{bmatrix}
	U_p\\
	Y_{p1}\\
	\0\\
	U_f
	\end{bmatrix}g_u=
	\begin{bmatrix}
	\0\\
	\0\\
	\0\\
	u-U_f g_\mathrm{ini}
	\end{bmatrix}.
	\end{align}
	Left multiplying both sides of~\eqref{eq.temp1} by $E^{-1}$, in view of~\eqref{eq.Eproperty1} and~\eqref{eq.Eproperty2}, one obtains
	\begin{align*}
	\begin{bmatrix}
	U_p\\
	Y_{p1}\\
	Y_{p2}\\
	U_f
	\end{bmatrix}g_u
	=
	\begin{bmatrix}
	\0\\
	\0\\
	\0\\
	u-U_f g_\mathrm{ini}
	\end{bmatrix}.
	\end{align*}
Furthermore, from the definition of $P_Y$, we have
	\begin{align*}
	\begin{bmatrix}
	U_p\\
	Y_p\\
	U_f
	\end{bmatrix}
	=
	\begin{bmatrix}
	\I & \0 & \0 \\
	\0 & P_Y^{-1}& \0 \\
	\0 & \0 & \I
	\end{bmatrix}
	\begin{bmatrix}
          U_p\\
          \begin{bmatrix}
	Y_{p1}\\
	Y_{p2}            
          \end{bmatrix}\\
	U_f
	\end{bmatrix}.
	\end{align*}
	Therefore, we conclude the proof by showing that
	\begin{align*}
	& \begin{bmatrix}
	U_p\\
	Y_p\\
	U_f
	\end{bmatrix}g_u=
	\begin{bmatrix}
	\I & \0 & \0 \\
	\0 & P_Y^{-1}& \0 \\
	\0 & \0 & \I
	\end{bmatrix}
	\begin{bmatrix}
          \0\\
          \begin{bmatrix}
	\0\\
	\0
          \end{bmatrix}\\
	u-U_f g_\mathrm{ini}
	\end{bmatrix}
	=
	\begin{bmatrix}
          \0\\
          \begin{bmatrix}
	\0\\
	\0
      \end{bmatrix}\\
	u-U_f g_\mathrm{ini}
	\end{bmatrix}.
	\end{align*}

\end{proof}

A solution $g=g_{\mathrm{ini}}+g_u$ to~\eqref{eq.yini-FirstThreeSolution} can be obtained from a $g_\mathrm{ini}$ verifying~\eqref{eq.yini-trajectoryconstraint} and the $g_u$ in Lemma~\ref{lemma.Yp2LinearCombination}.
We next show that $y$ can be expressed as an affine function of $u$ and $g_w$, and further reformulate the tracking error constraint in terms of $u$ and $g_w$.
\begin{lemma}\label{lem.y_LQTE_parameterization}
	Given a feasible noise trajectory $w$ and a control sequence $u$, the unique output $y$ satisfying~\eqref{eq.yini-model-minmax} is given by 
	\begin{equation}\label{eq.y_parameterization}
		y = B_uu+B_wg_w+y_0,
	\end{equation}
	where $g_w$ parameterizes $w$ through~\eqref{eq.w_parameterization}, $y_0 = B_\mathrm{ini}g_{\mathrm{ini}}^*$,  $B_w = B_\mathrm{ini}M$,
	\begin{align*}
		B_\mathrm{ini}& = Y_f\left(\I+\Lambda^\top (\Lambda\Lambda^\top)^{-1}
		\begin{bmatrix}
			\0\\
			\0\\
			-U_f
		\end{bmatrix}\right),\\
	B_u& = Y_f \Lambda^\top (\Lambda\Lambda^\top)^{-1}
		\begin{bmatrix}
			\0\\
			\0\\
			\I
		\end{bmatrix}.
	\end{align*}
	Moreover,  the tracking error constraint $\mathrm{LQTE}(u,y,w)\leq\gamma$ can be equivalently expressed as
	\begin{align}
		\begin{bmatrix}
			1 \\
			g_w
		\end{bmatrix}^\top
		\underbrace{\begin{bmatrix}
			{[Q_g]}_{11} & {[Q_g]}_{12}\\
			{[Q_g]}_{12}^\top & {[Q_g]}_{22}
			\end{bmatrix}}_{Q_g(u,\gamma)}
		\begin{bmatrix}
			1\\
			g_w
		\end{bmatrix}\ge 0, \label{yini-performancequadraticform}
	\end{align}
where 
\begin{align*}
   &\bar{R}=\I\otimes R, \quad        \bar{Q}=\I\otimes Q,\\
	&{[Q_g]}_{11}=\gamma-u^\top\bar{R}u-(B_uu+y_0-r)^\top\bar{Q}(B_uu+y_0-r) ,\\
	&{[Q_g]}_{12}= -(B_uu+y_0-r)^\top\bar{Q}B_w,\quad  {[Q_g]}_{22}= -B_w^\top\bar{Q}B_w.
\end{align*}
\end{lemma}

\begin{proof}
Substituting $g=g_\mathrm{ini}+g_u$ into the fourth block row of~\eqref{eq.yini-model-minmax}, one obtains
	\begin{align*}
	y&=Y_f g_\mathrm{ini}+Y_f g_u\\
	&=Y_f g_\mathrm{ini}+Y_f\Lambda^\top (\Lambda\Lambda^\top)^{-1}
	\left(\begin{bmatrix}
	\0\\
	\0\\
	\I
	\end{bmatrix}u+
	\begin{bmatrix}
	\0\\
	\0\\
	-U_f
	\end{bmatrix}g_\mathrm{ini} \right)\\
	&=B_\mathrm{ini} g_\mathrm{ini} +B_uu = B_\mathrm{ini}(g_{\mathrm{ini}}^*+Mg_w)+B_uu,
	\end{align*}
which proves~\eqref{eq.y_parameterization}. 
We then have the following equivalent conditions
\begin{align*}
  &\text{LQTE}(u,y,w)\leq \gamma \Leftrightarrow \gamma - u^\top\bar{R}u-(y-r)^\top\bar{Q}(y-r)\geq 0,\\
  & \Leftrightarrow \gamma - u^\top\bar{R}u\\
  &  -(y_0+B_wg_w+B_uu-r)^\top\bar{Q}(y_0+B_wg_w+B_uu-r)\ge 0,\\
  & \Leftrightarrow \eqref{yini-performancequadraticform}.
\end{align*}
The proof is complete.
\end{proof}
\subsection{Main Result} \label{subsec.MainResult}

The following theorem leverages the results obtained in Lemma~\ref{lem.g_w_parameterization} and Lemma~\ref{lem.y_LQTE_parameterization} to show that~\eqref{eq.minimization_problem} and, hence, \textbf{P1}, are equivalent to a minimization problem with a linear cost and LMI constraints.

\begin{theorem}\label{thm.LMIopt}
  The robust tracking control problem \textbf{P1} is equivalent to solving
  \begin{subequations} \label{eq.finalOpt}
\begin{gather}
  \min_{u, \gamma, \alpha\ge 0}\; \gamma\\
  \mathrm{s.t.},
\begin{bmatrix}
{(\bar{R}+B_u^\top\bar{Q}B_u)}^{-1}&
\begin{bmatrix}
u & \0      
\end{bmatrix}\\
\begin{bmatrix}
u^\top\\
\0
\end{bmatrix}& Q_g^a(u,\gamma)-\alpha A_w
\end{bmatrix}\ge 0,\label{eq.finalLMI}
\end{gather}
  \end{subequations}
 where
 \begin{equation}\label{eq.Qgadef}
 	Q_g^a(u, \gamma)=Q_g(u, \gamma)+\begin{bmatrix}
 	u^\top (R+B_u^\top\bar{Q}B_u)u & \0 \\
 	\0 & \0
 	\end{bmatrix},
      \end{equation}
 $A_w$ is defined in Lemma~\ref{lem.g_w_parameterization}, and $\bar{Q},~\bar{R},~Q_g,$ and $B_u$ are defined in Lemma~\ref{lem.y_LQTE_parameterization}.
\end{theorem}

\begin{proof}
  Based on Lemma~\ref{lem.g_w_parameterization} and Lemma~\ref{lem.y_LQTE_parameterization}, the minimization problem~\eqref{eq.minimization_problem} is equivalent to 
  \begin{gather*}
  \min_{u,\gamma} \quad \gamma \\
  \text{s.t.},  \quad \eqref{yini-performancequadraticform} ~\text{holds}~ \forall g_w \text{ satisfying } \eqref{eq.g_w_constraint}.
  \end{gather*}
  In view of the S-lemma~\cite{polik2007survey}, the constraint of this minimization problem holds if and only if there exist $u$ and $\alpha\ge 0$ such that
  \begin{equation*}
  Q_g(u, \gamma)-\alpha A_w\ge 0.
  \end{equation*}
  Using Schur complement~\cite{boyd1994linear}, the above matrix inequality can be transformed into the LMI in~\eqref{eq.finalLMI}.
  Note that the quadratic term $u^\top (R+B_u^\top\bar{Q}B_u)u$ in the right hand side of~\eqref{eq.Qgadef} cancels out with the quadratic term of $u$ in $Q_g(u, \gamma)$, therefore making $Q_g^a(u, \gamma)$ a linear function of $u$ and $\gamma$.
  Minimizing the performance index $\gamma$ further gives the solution of~\eqref{eq.minimization_problem} and hence \textbf{P1}.
\end{proof}

\subsection{ Implementation Aspects: Dimension Reduction for Computational Efficiency}\label{subsec.ReducingDimension}
In view of the analysis in Section~\ref{subsec.NoiseConstTransform}, the sequence $w$ makes $(u_{\mathrm{ini}}, y_{\mathrm{ini}}-w)$ a trajectory of $\mathcal{G}$ if and only if there exists $g_w$, such that 
\begin{align}
  \label{eq:nonRedMap}
w=-Y_p\mathcal{N}(U_p)g_w+w_0.
\end{align}
However, if $g_w$ is mapped into $w$ through~\eqref{eq:nonRedMap}, any $g_w+v$, where $v\in \mathrm{ker}(Y_p\mathcal{N}(U_p))$, is also mapped into the same $w$.
This is especially true when the length $T_d$ of historical data is large, i.e., $T_d\gg mT_\mathrm{ini}$ and $T_d\gg pT_\mathrm{ini}$, which makes $\mathrm{ker}(U_p) \cap \mathrm{ker}(Y_p)$, and therefore $\mathrm{ker}(Y_p\mathcal{N}(U_p))$, nonempty.
As a result, any parameterization of a subspace through $g_w$ is redundant.
Redundancy affects also the constraint~\eqref{eq.finalLMI}. 
Indeed, if the length of the vector $g_w$ is unnecessarily large, so are the sizes of the matrices $A_w$ and $Q_g(u,\gamma)$ in~\eqref{eq.g_w_constraint},~\eqref{yini-performancequadraticform}, as well as the LMI constraint in~\eqref{eq.finalLMI}, making the optimization problem~\eqref{eq.finalOpt} inefficient. 

More formally, denote $\mathcal{W}$ as the set of noise trajectories $w$ that make $(u_{\mathrm{ini}}, y_{\mathrm{ini}}-w)$ a trajectory of $\mathcal{G}$.
We have, from~\eqref{eq:nonRedMap}, $\mathcal{W}=\mathrm{range}(Y_p\mathcal{N}(U_p))+w_0$, where we represent the vector space $\mathrm{range}(Y_p\mathcal{N}(U_p))$ as
\begin{align}
  \label{eq:1}
\{Y_p\mathcal{N}(U_p)g_w| g_w \in \mathbb{R}^{T_d-(m+1)T_{\mathrm{ini}}-T_f+1}\}.  
\end{align}
 The cause of redundancy is that the dimension of the free vector $g_w$ in~\eqref{eq:1} can be much larger than the dimension of  $\mathrm{range}(Y_p\mathcal{N}(U_p))$.
In the following theorem, we  address this issue to present a non-redundant representation of $\mathcal{W}$.
\begin{theorem}\label{thm:M_selection}
 The vector $w$ belongs to $\mathcal{W}$ if and only if there exists $g_w\in \mathbb{R}^{\bar{n}_w}$ such that
	\begin{equation}\label{eq.M_selection}
	w=-Y_p\mathcal{N}(U_p) \mathcal{R}(\mathcal{N}(U_p)^\top Y_p^\top)g_w+w_0.
      \end{equation}
      where $\bar{n}_w={\mathrm{rank}(Y_p\mathcal{N}(U_p))}$.
        Moreover, the above mapping from $\mathbb{R}^{\bar{n}_w}$ to $\mathcal{W}$ is bijective.
\end{theorem}

\begin{proof}
  Since  two vector spaces are isomorphic if and only if they have the same dimension, to eliminate the redundant representation problem, we introduce an isomorphism from  $\mathbb{R}^{\bar{n}_w}$ to  $\mathrm{range}(Y_p\mathcal{N}(U_p))$ and represent $\mathrm{range}(Y_p\mathcal{N}(U_p))$ in terms of this isomorphism.
Notice that
\begin{align*}
  &  \mathrm{range}(Y_p\mathcal{N}(U_p))=\{Y_p\mathcal{N}(U_p)g| g\in \mathbb{R}^{T_d-T_{\mathrm{ini}}+1-T_{\mathrm{ini}}m}\} \\
  &\overset{(a)}{=}\{Y_p\mathcal{N}(U_p)(g_1+g_2)| g_1\in \mathrm{ker}(Y_p\mathcal{N}(U_p)),\\
  & \qquad \qquad \qquad \qquad \quad \qquad \qquad \qquad g_2\in \mathrm{range}(\mathcal{N}(U_p)^\top Y_p^\top)\}\\
                  &=\{Y_p\mathcal{N}(U_p)g_2|  g_2\in \mathrm{range}(\mathcal{N}(U_p)^\top Y_p^\top)\}\\
  &=\{Y_p\mathcal{N}(U_p) \mathcal{R}(\mathcal{N}(U_p)^\top Y_p^\top)g_w| g_w\in \mathbb{R}^{\bar{n}_w}\}
\end{align*}
where $(a)$ follows from the fact that
\begin{align*}
\mathrm{ker}(Y_p\mathcal{N}(U_p)) \perp \mathrm{range}(\mathcal{N}(U_p)^\top Y_p^\top).  
\end{align*}
Therefore an isomorphism from $\mathbb{R}^{\bar{n}_w}$  to   $\mathrm{range}(Y_p\mathcal{N}(U_p))$ is given by the matrix $Y_p\mathcal{N}(U_p) \mathcal{R}(\mathcal{N}(U_p)^\top Y_p^\top)$.
Furthermore, since $\mathcal{W}$ is $\mathrm{range}(Y_p\mathcal{N}(U_p))$ shifted by $w_0$, the mapping from  $\mathbb{R}^{\bar{n}_w}$ to $\mathcal{W}$  given by~\eqref{eq.M_selection} is bijective.
\end{proof}
In view of the above theorem, the representation of $\mathcal{W}$ through~\eqref{eq.M_selection} using $g_w\in \mathbb{R}^{\bar{n}_w}$ is non-redundant.
To apply the above result in the implementation of~\eqref{eq.finalOpt}, we only need to replace the matrix $M=\mathcal{N}(U_p)$ in the derivations of Section~\ref{subsec.NoiseConstTransform}--\ref{subsec.MainResult} with $M=\mathcal{N}(U_p) \mathcal{R}(\mathcal{N}(U_p)^\top Y_p^\top)$.

\begin{remark}\label{rem:Mg_w_dimension}
  Since
  \begin{align*}
{\bar{n}_w}=  \mathrm{rank}\left(\begin{bmatrix}
	U_p \\
	Y_p
      \end{bmatrix}\right)-T_{\mathrm{ini}}m \overset{(a)}{=} n,  
  \end{align*}
  where $(a)$ follows from Theorem 2 of~\cite{RN11380}, the length of the vector $g_w$ in~\eqref{eq.M_selection} is equal to $n$.
    This guarantees that the size of the matrix in the LMI in~\eqref{eq.finalLMI} scales with $n+T_fm$.
    On the contrary, the length of the vector $g_w$ in~\eqref{eq.w_parameterization} scales with $T_d$.
    As $T_d$ is usually significantly larger than $n$, the non-redundant parameterization shown in this section can reduce the size of the LMI constraint~\eqref{eq.finalLMI} considerably.
\end{remark}

\section{Generalizations}\label{sec:Extensions}
In this section, we provide several extensions to the robust control design method described in the previous section.
First, in Section~\ref{subsec.IOConstraints}, we show how to add input and output constraints to the controller.
In Section~\ref{subsec.InputDist}, we show how to take into account actuator disturbances, before presenting the overall LMI optimization problem incorporating both extensions in Section~\ref{subsec.FinalControlDesignProblem}.

\subsection{Input and Output Constraints} \label{subsec.IOConstraints}
In this section, we show how to add quadratic input and output constraints to problem \textbf{P1}.
Since constraints on the input can be directly incorporated into~\eqref{eq.finalOpt}, hereafter we focus on constraints on the output $y$ only and in the form
	\begin{align}\label{eq.O_cons_quad}
	\theta(y)=\begin{bmatrix}
	1 \\
	y
	\end{bmatrix}^\top\underbrace{\begin{bmatrix}
		\Theta_{11} & \Theta_{12} \\
		\Theta_{12}^\top & \Theta_{22}
		\end{bmatrix}}_{\Theta}\begin{bmatrix}
	1 \\
	y
	\end{bmatrix}\geq 0.
	\end{align}
When $\Theta_{22}=\Theta_{22}^\top<0$ and $\Theta_{12}=0$, the above constraint imposes an upper bound on $\|y\|_2$.

Since the output is related to the input $u$ and the noise trajectory $w$ via~\eqref{eq.y_parameterization}, the output constraint~\eqref{eq.O_cons_quad} can be written as 
\begin{equation*}
	\begin{split}
	 \Theta_{11}&+\Theta_{12}\left(y_0+B_uu+B_wg_w\right)+\left(y_0+B_uu+B_wg_w\right)^\top\Theta_{12}^\top \\
	 &+\left(y_0+B_uu+B_wg_w\right)^\top\Theta_{22}\left(y_0+B_uu+B_wg_w\right) \\
	 &= \begin{bmatrix}
	 1 \\
	 g_w
	 \end{bmatrix}^\top \underbrace{\begin{bmatrix}
	 	\left[\Theta_g\right]_{11} & \left[\Theta_g\right]_{12} \\
	 	\left[\Theta_g\right]_{12}^\top & \left[\Theta_g\right]_{22}
	 	\end{bmatrix} }_{\Theta_g} \begin{bmatrix}
	 1 \\
	 g_w
	 \end{bmatrix} \geq 0,
	\end{split}
\end{equation*}
where
\begin{align*}
  \left[\Theta_g\right]_{22}& = B_w^\top\Theta_{22}B_w,\\
  \left[\Theta_g\right]_{11} &= \Theta_{11} + \Theta_{12}\left(y_0+B_uu\right)+\left(y_0+B_uu\right)^\top\Theta_{12}^\top\\
                            &+ \left(y_0+B_uu\right)^\top\Theta_{22}\left(y_0+B_uu\right),\\
  \left[\Theta_g\right]_{12} &= \Theta_{12}B_w+\left(y_0+B_uu\right)^\top\Theta_{22}B_w .
\end{align*}
In principle, we want  the constraint~\eqref{eq.O_cons_quad} to hold for all feasible noise trajectories.
Similarly to the proof of Theorem~\ref{thm.LMIopt}, this requirement is equivalent to the existence of $\alpha_y\geq0$ such that
$\Theta_g-\alpha_yA_w\geq 0$,
which can be reformulated as an LMI constraint and added to the optimization problem~\eqref{eq.finalOpt}.

\subsection{Actuator Disturbances}\label{subsec.InputDist}
In this section, we show how to consider actuator disturbances.
Suppose  the actuation input $\check{u}_{\mathrm{ini}}$ to the system $\mathcal{G}$ to generate the recent data is also noisy, i.e., 
$$\check{u}_\mathrm{ini}=u_\mathrm{ini}-d_\mathrm{ini},$$
where $u_\mathrm{ini}$ is nominal control input and $d_\mathrm{ini}$ is the actuator disturbance.
Moreover, we also consider a disturbance $d$ acting on the computed control input $u$, i.e., $\check{u}= u-d$.
Therefore, the data-dependent relation~\eqref{eq.yini-model-minmax} becomes

\begin{equation}\label{eq.yini-model-disturbance}
	\begin{bmatrix}
	U_p \\
	Y_p \\
	U_f \\
	Y_f
	\end{bmatrix} g=\begin{bmatrix}
	u_{\mathrm{ini}} \\
	y_{\mathrm{ini}} \\
	u \\
	y
	\end{bmatrix}-\begin{bmatrix}
	d_\mathrm{ini} \\
	w \\
	d \\
	\0
	\end{bmatrix}.
\end{equation}

We assume that the actuator disturbance $\bar{d}\triangleq [d_\mathrm{ini}^\top,d^\top]^\top$  satisfies the quadratic constraint
\begin{equation}\label{eq.distconstraint}
	\begin{bmatrix}
	1 \\
	\bar{d}
	\end{bmatrix}^\top\underbrace{\begin{bmatrix}
		\Phi_{d,11} & \Phi_{d,12} \\
		\Phi_{d,12}^\top & \Phi_{d,22}
		\end{bmatrix}}_{\Phi_{d}}\begin{bmatrix}
	1 \\
	\bar{d}
	\end{bmatrix}\geq 0,
\end{equation}
where $\Phi_{d, 22}< 0$.

Our goal is to solve a min-max robust control problem similar to \textbf{P1}.
Due to the existence of actuator disturbances, we replace $\|u\|_R^2$ with the true system input $\|\check{u}\|_R^2$ in the cost~\eqref{eq.yini-performanceobjective}, replace~\eqref{eq.yini-model-minmax} with~\eqref{eq.yini-model-disturbance}, and optimize over all feasible noise and disturbance trajectories $[d_\mathrm{ini}^\top, w^\top, d^\top]^\top$.
We first characterize feasible trajectories $[d_\mathrm{ini}^\top, w^\top]^\top $ such that $(u_{\mathrm{ini}}-d_{\mathrm{ini}}, y_{\mathrm{ini}}-w)$ is a trajectory of $\mathcal{G}$.
Similarly to the argument in Section~\ref{subsec.NoiseConstTransform}, $[d_\mathrm{ini}^\top, w^\top]^\top $ satisfies the above requirement if and only if there exists $g_{\mathrm{ini}}$ such that
\begin{equation}\label{eq.g_1uini}
  \begin{bmatrix}
    d_\mathrm{ini}\\
    w
  \end{bmatrix}=
	-\begin{bmatrix}
	U_p \\
	Y_p
      \end{bmatrix}
      g_\mathrm{ini} + \begin{bmatrix}
	u_\mathrm{ini} \\
	y_\mathrm{ini} \\
	\end{bmatrix}.
\end{equation}
Therefore, the set of noise and disturbance trajectories $[d_\mathrm{ini}^\top, w^\top]^\top $ that make $(u_\mathrm{ini} - d_\mathrm{ini}, y_\mathrm{ini}- w )$ a trajectory of $\mathcal{G}$ is
$$\tilde{\mathcal{W}}=\mathrm{range}(\begin{bmatrix}
	U_p \\
	Y_p
      \end{bmatrix}
)+\begin{bmatrix}
	u_\mathrm{ini}  \\
	y_\mathrm{ini} \\
	\end{bmatrix}.
        $$
Let  $\bar{n}_d=\mathrm{rank}(\left[\begin{smallmatrix}
	U_p  \\
	Y_p \\
	\end{smallmatrix}\right]
      )=\bar{n}_w+mT_\mathrm{ini}$.
      Similarly to the analysis in Section~\ref{subsec.ReducingDimension},  $[d_\mathrm{ini}^\top, w^\top]^\top \in\tilde{\mathcal{W}}$ if and only if there exists $g_w\in \mathbb{R}^{\bar{n}_d}$ such that
	\begin{equation}\label{eq.M_selection2}
          \begin{bmatrix}
            d_\mathrm{ini}\\
            w
          \end{bmatrix}
	=-\begin{bmatrix}
	U_p \\
	Y_p
      \end{bmatrix}
 \mathcal{R}(\begin{bmatrix}
	U_p \\
	Y_p
      \end{bmatrix}^\top)g_w+
\begin{bmatrix}
	u_\mathrm{ini}  \\
	y_\mathrm{ini} \\
	\end{bmatrix}.
	\end{equation}
        Moreover, the above mapping from $\mathbb{R}^{\bar{n}_d}$ to $\tilde{\mathcal{W}}$ is bijective.
                Therefore, from~\eqref{eq.M_selection2}, one gets 
\begin{equation}\label{eq.rho_parameterization}
	 \begin{bmatrix}
          \begin{bmatrix}
          d_{\mathrm{ini}} \\
          w
        \end{bmatrix}\\
	d
      \end{bmatrix} =
      \begin{bmatrix}
- \begin{bmatrix}
	U_p \\
	Y_p
      \end{bmatrix}
 \mathcal{R}(\begin{bmatrix}
	U_p \\
	Y_p
      \end{bmatrix}^\top)	 & \0 \\
		\0 & \I \\
              \end{bmatrix}
              \underbrace{\begin{bmatrix}
		g_w \\
		d 
		\end{bmatrix}}_{\bar{g}}+\begin{bmatrix}
                \begin{bmatrix}
		u_\mathrm{ini} \\
                		y_\mathrm{ini}                  
                \end{bmatrix}\\
		\0 
		\end{bmatrix},
\end{equation}
i.e., the vector $\bar{g}\in\mathbb{R}^{\bar{n}_d+mT_f}$ parameterizes all feasible noise and disturbance trajectories.
By following the arguments used in the proof of Lemma~\ref{lem.g_w_parameterization}, the quadratic constraints on $w$ and $\bar{d}$ can be transformed into quadratic constraints on $\bar{g}$ as
\begin{gather}\label{eq.gbar_const}
	\begin{bmatrix}
	1 \\
	w
	\end{bmatrix}^\top \Phi \begin{bmatrix}
	1 \\
	w
	\end{bmatrix}\geq 0 \iff 
	\begin{bmatrix}
	1 \\
	\bar{g}
	\end{bmatrix}^\top \bar{\Phi}_w \begin{bmatrix}
	1 \\
	\bar{g}
	\end{bmatrix}\geq 0 ,\\
	\begin{bmatrix}
	1 \\
	\bar{d}
	\end{bmatrix}^\top \Phi_d \begin{bmatrix}
	1 \\
	\bar{d}
	\end{bmatrix}\geq 0 \iff 
	\begin{bmatrix}
	1 \\
	\bar{g}
	\end{bmatrix}^\top \bar{\Phi}_d \begin{bmatrix}
	1 \\
	\bar{g}
	\end{bmatrix}\geq 0 ,\label{eq.gbar_const2}
\end{gather}
where the matrices $\bar{\Phi}_w$ and $\bar{\Phi}_d$ directly follow from~\eqref{eq.yini-noiseconstraint}, \eqref{eq.distconstraint}, and \eqref{eq.rho_parameterization}, and their expressions are omitted for brevity. 

Since every $[d_\mathrm{ini}^\top, w^\top]^\top \in \tilde{\mathcal{W}}$ can be written as~\eqref{eq.M_selection2}, by substituting this representation into~\eqref{eq.g_1uini}, we obtain that for a given $[d_\mathrm{ini}^\top, w^\top]^\top$, the solution $g_{\mathrm{ini}}$ to~\eqref{eq.g_1uini} is given by $  g_{\mathrm{ini}}=M_dg_w$, where $M_d=\mathcal{R}(\left[\begin{smallmatrix}
	U_p \\
	Y_p
\end{smallmatrix}\right]^\top)$.
We can follow the procedure in Section~\ref{subsec.PerformanceTransformation} to derive the solution $g = g_\mathrm{ini}+g_u$ to the first three equations in~\eqref{eq.yini-model-disturbance}, where $g_u=\Lambda^\top (\Lambda\Lambda^\top)^{-1}\begin{bmatrix}
\mathbf{0}&\mathbf{0}&(\check{u}-U_fg_\mathrm{ini})^\top
\end{bmatrix}^\top$ verifies~\eqref{eq.g_u} with the noisy control input $\check{u}$ instead of $u$. 
Then, since $y=Y_fg$, the following holds with the matrices $B_\mathrm{ini}$ and $B_u$ defined in Lemma~\ref{lem.y_LQTE_parameterization}
\begin{align}
y &= B_\mathrm{ini}g_\mathrm{ini} + B_u\check{u} = B_\mathrm{ini}M_dg_w-B_ud+B_uu \nonumber \\
&= \underbrace{\begin{bmatrix}
  B_\mathrm{ini}M_d&-B_u
\end{bmatrix}}_{\bar{B}_g}\bar{g} + B_uu.
\end{align}
Since $\check{u} = u-\Xi\bar{g},$ where $\Xi = [\0, \I]$, the performance constraint can be rewritten as
\begin{align*}
&\gamma-  \sum_{k=0}^{T_f-1}\left(\left\|y_{k}-r_{t+k}\right\|_{Q}^{2}+\left\|\check{u}_{k}\right\|_{R}^{2}\right) \\
& = \gamma - \check{u}^\top\bar{R}\check{u}-(y-r)^\top\bar{Q}(y-r)\\
&=  \gamma - \left(u-\Xi\bar{g}\right)^\top\bar{R}\left(u-\Xi\bar{g}\right)\\
&-(\bar{B}_g\bar{g}+B_uu-r)^\top\bar{Q}(\bar{B}_g\bar{g}+B_uu-r)\\
&=  \begin{bmatrix}
1 \\
\bar{g}
\end{bmatrix}^\top
\underbrace{\begin{bmatrix}
	[\bar{Q}_g]_{11} & [\bar{Q}_g]_{12} \\
	[\bar{Q}_g]_{12}^\top & [\bar{Q}_g]_{22}   
	\end{bmatrix}}_{\bar{Q}_g(u,\gamma)}
\begin{bmatrix}
1\\
\bar{g}
\end{bmatrix}\ge 0,
\end{align*}
where
\begin{equation*}
	\begin{split}
	[\bar{Q}_g]_{11} &= \gamma - u^\top\bar{R}u-\left(B_uu-r\right)^\top\bar{Q}\left(B_uu-r\right), \\
	[\bar{Q}_g]_{12} &= u^\top\bar{R}\Xi -\left(B_uu-r\right)^\top\bar{Q}\bar{B}_g, \\
	[\bar{Q}_g]_{22} &= -\Xi^\top\bar{R}\Xi - \bar{B}_g^\top\bar{Q} \bar{B}_g.
	\end{split}
\end{equation*}

As such, the overall data-driven robust control objective is to find $u$ and $\gamma$ such that $$\begin{bmatrix}
1 \\
\bar{g}
\end{bmatrix}^\top\bar{Q}_g(u,\gamma)\begin{bmatrix}
1 \\
\bar{g}
\end{bmatrix}\geq 0$$
holds for all feasible noise and disturbance trajectories parameterized by $\bar{g}$ satisfying  quadratic constraints~\eqref{eq.gbar_const},~\eqref{eq.gbar_const2}.
Using the S-lemma for multiple quadratic inequalities~\cite{polik2007survey}, this is true if there exist $u$, $\alpha_w\geq0$, and $\alpha_d\geq0$ such that
\begin{align}
\bar{Q}_g(u,\gamma)-\alpha_w\bar{\Phi}_w-\alpha_d\bar{\Phi}_d\geq 0.
  \label{eq.finalLMI_dist}
\end{align}
We can further convert the above inequality into an LMI through the Schur complement.
Therefore, the problem \textbf{P1} with input disturbances is solved if the following optimization problem is solved
\begin{equation*} \label{eq.finalOpt_dist}
\begin{split}
\min_{u, \gamma, \alpha_w\geq 0,\alpha_d\geq0}\; \gamma \quad \text{s.t.},~\eqref{eq.finalLMI_dist}.
\end{split}
\end{equation*}

\vspace{-0.3cm}\subsection{Co-existence of Quadratic Input/Output Constraints and Actuator Disturbance}\label{subsec.FinalControlDesignProblem}

In this section, we use the results in Sections~\ref{subsec.IOConstraints} and~\ref{subsec.InputDist} for dealing simultaneously with the quadratic input/output constraints and actuator disturbances.
The overall robust control problem is given by
\begin{subequations}
  \label{eq.overall_minimax}
\begin{align}
\min_u \max_{w,d_\mathrm{ini},d}&\quad \sum_{k=0}^{T_f-1}\left(\left\|y_{k}-r_{k}\right\|_{Q}^{2}+\left\|u_{k}-d_k\right\|_{R}^{2}\right)  \\
\text{subject to}&\quad \begin{bmatrix}
U_p \\
Y_p \\
U_f \\
Y_f
\end{bmatrix} g=\begin{bmatrix}
u_{\mathrm{ini}} \\
y_{\mathrm{ini}} \\
u \\
y
\end{bmatrix}-\begin{bmatrix}
d_\mathrm{ini} \\
w \\
d \\
\0
\end{bmatrix}, \\
& \quad \begin{bmatrix}
	1 \\
	u-d
	\end{bmatrix}^\top\underbrace{\begin{bmatrix}
		\Psi_{11} & \Psi_{12} \\
		\Psi_{12}^\top & \Psi_{22}
		\end{bmatrix}}_{\Psi}\begin{bmatrix}
	1 \\
	u-d
	\end{bmatrix}\geq 0, \label{eq.I_cons_quad} \\
& \quad \text{output quadratic constraints }\eqref{eq.O_cons_quad}.
\end{align}
\end{subequations}
The following theorem provides an LMI-based representation of~\eqref{eq.overall_minimax}\footnote{Even though the matrix inequalities in the theorem and proof are not LMIs, they can be transformed to LMIs using Schur complement in a similar way to the proof of Theorem~\ref{thm.LMIopt}. To save space, the resulting LMIs are not displayed. Furthermore, we refer to these matrix inequalities as LMIs without ambiguity.}.
 
 \begin{theorem}\label{thm.LMIopt_overall}
Denote $\alpha\triangleq [\alpha_w,\alpha_d,\alpha_{u,w},\alpha_{u,d},\alpha_{y,w},\alpha_{y,d}]^\top$. The optimization problem~\eqref{eq.overall_minimax} is solved when the following minimization problem is solved
 	\begin{equation} \label{eq.finalOpt_overall}
 	\min_{u,\gamma, \alpha\geq 0}\; \gamma \quad
 	\text{s.t.},~\eqref{eq.finalLMI_dist},~\eqref{eq.I_cons_dist_LMI},~\eqref{eq.O_cons_dist_LMI}.
      \end{equation}
      where constraints~\eqref{eq.I_cons_dist_LMI} and~\eqref{eq.O_cons_dist_LMI} are given in the proof.
 \end{theorem} 

 \begin{proof}
   Similarly to the proof of Theorem~\ref{thm.LMIopt}, we aim to minimize $\gamma$, subject to the tracking error constraint and the constraint that the input/output constraints hold for all feasible noise and disturbance trajectories.
   The tracking error constraint can be shown to be given as~\eqref{eq.finalLMI_dist}.
   In the following, we show how to characterize  the constraint that the input/output constraints hold for all feasible noise and disturbance trajectories.

   In light of $\check{u}=u-\Xi\bar{g}$, one sees that the input constraint in~\eqref{eq.I_cons_quad} is equivalent to
 
 \begin{equation}\label{eq.I_cons_dist}
 	\bar{\psi}(\bar{g})=\begin{bmatrix}
 	1 \\
 	\bar{g}
 	\end{bmatrix}^\top\underbrace{\begin{bmatrix}
 		[\bar{\Psi}]_{11} & [\bar{\Psi}]_{12} \\
 		[\bar{\Psi}]_{12}^\top & [\bar{\Psi}]_{22}
 		\end{bmatrix}}_{\bar{\Psi}}\begin{bmatrix}
 	1 \\
 	\bar{g}
 	\end{bmatrix}\geq 0,
 \end{equation}
 where 
 \begin{equation}\label{eq.barPsi}
 	\begin{split}
 	[\bar{\Psi}]_{11} &= \Psi_{11}+u^\top\Psi_{22}u+\Psi_{12}u+u^\top\Psi_{12}^\top, \\
 	[\bar{\Psi}]_{12} &=- \Psi_{12}\Xi-u^\top\Psi_{22}\Xi, \qquad [\bar{\Psi}]_{22} = \Xi^\top\Psi_{22}\Xi.
 	\end{split}
 \end{equation}
  We need to ensure that~\eqref{eq.I_cons_dist} holds for all $\bar{g}$ satisfying the quadratic constraints~\eqref{eq.gbar_const}, \eqref{eq.gbar_const2}.
 In view of the S-lemma, this is possible if there exist $u$, $\alpha_{u,w}\geq 0$, and $\alpha_{u,d}\geq 0$ such that 
 \begin{align}~\label{eq.I_cons_dist_LMI}
\bar{\Psi}-\alpha_{u,w}\bar{\Phi}_w-\alpha_{u,d}\bar{\Phi}_d\geq 0,   
 \end{align}
 which can be converted to an LMI using the Schur complement.
 Similarly, considering $y = \bar{B}_g\bar{g} + B_uu$, the output constraint in~\eqref{eq.O_cons_quad} is equivalent to
\begin{equation}\label{eq.O_cons_dist}
\bar{\theta}(\bar{g})=\begin{bmatrix}
1 \\
\bar{g}
\end{bmatrix}^\top\underbrace{\begin{bmatrix}
	[\bar{\Theta}]_{11} & [\bar{\Theta}]_{12} \\
	[\bar{\Theta}]_{12}^\top & [\bar{\Theta}]_{22}
	\end{bmatrix}}_{\bar{\Theta}}\begin{bmatrix}
1 \\
\bar{g}
\end{bmatrix}\geq 0,
\end{equation}
where 
\begin{equation}\label{eq.barTheta}
\begin{split}
[\bar{\Theta}]_{11} &= \Theta_{11}+u^\top B_u^\top\Theta_{22}B_uu+\Theta_{12}B_uu+u^\top B_u^\top\Theta_{12}^\top, \\
[\bar{\Theta}]_{12} &= \Psi_{12}\bar{B}_g+u^\top B_u^\top \Theta_{22}\bar{B}_g, \qquad [\bar{\Theta}]_{22} = \bar{B}_g^\top\Theta_{22}\bar{B}_g.
\end{split}
\end{equation}
Following the same procedure as for input constraint, it can be shown that~\eqref{eq.O_cons_dist} is satisfied if there exist $u$, $\alpha_{y,w}\geq 0$, and $\alpha_{y,d}\geq 0$ such that the following is satisfied
\begin{equation}\label{eq.O_cons_dist_LMI}
\bar{\Theta} -\alpha_{y,w}\bar{\Phi}_w-\alpha_{y,d}\bar{\Phi}_d\geq 0,   
\end{equation}
which can be converted to an LMI using the Schur complement.
Combining the above results, we get~\eqref{eq.finalOpt_overall}.
\end{proof} 

\begin{remark}
	The extensions presented in this section involve the use of the S-lemma with multiple quadratic constraints~\cite{polik2007survey} and Schur complement with semidefinite matrices~\cite{boyd1994linear}, which are only sufficient. Therefore, the control design procedure in~\eqref{eq.finalOpt_overall} is conservative, i.e., it may have no solution even though a control input $u$ solving the min-max control problem~\eqref{eq.overall_minimax} exists.
\end{remark}

\begin{remark}
  The proposed control design can be easily applied in a receding horizon fashion, in order to implement a data-driven predictive controller. In doing so, at each time instance, one needs to update the output reference $r$ as well as recent input and output data $u_\mathrm{ini}$ and $y_\mathrm{ini}$ with the online data, solve~\eqref{eq.finalOpt_overall}, and apply only the first control input from the computed optimal control sequence $u$. Moreover, it can be shown that the resulting controller is equivalent to a robust model predictive controller (MPC) with bounded uncertainty on the initial state. As such, the stability of the resulting closed-loop system can be studied using the existing results on robust MPC. Such a discussion is omitted so as to emphasize the robust data-driven nature of the proposed controller, which is the main contribution of this paper.
\end{remark}


\section{Simulations}\label{subsec.Simulation}


We illustrate the performance of the proposed controller through numerical simulations.
We consider an unstable LTI system~\eqref{eq.LTI_system_ss} with randomly selected system matrices 
\begin{equation*}
	\begin{split}
	A &= \begin{bmatrix}
	0.6799 & -0.0331 & -0.8332 & 0.4924\\
	0.9748 & 1.0060 & 0.3666 & 0.5863\\
	0.7311 & 0.3693 & -1.0711 & 0.1603\\
  -0.7442 & 0.0330 & 0.0667 & 0.1961
	\end{bmatrix}, \\
	 B &= \begin{bmatrix}
	-0.7841 & -0.1798 & -0.0757\\
	  0.5204 & -0.5806 & -0.6510\\
	  0.1974 & 0.2140 & -0.4851\\
	-0.9378 & 0.7881 & -0.1826
	\end{bmatrix}, \\
	C &= \begin{bmatrix}
	0.4458 & 0.4911 & 0.7394 & -0.1359\\
	0.0733 & -0.1468 & -0.6357 & 0.7353
	\end{bmatrix}, \enskip D=\0.
	\end{split}
\end{equation*}
By choosing a random initial condition, historical input-output data of length $T_d=110$ are collected with inputs generated from a uniform distribution in the interval $[-1,1]$. We assume that the exact order $n=4$ of the system is unknown and only the upper bound $\bar{n}=6$ is available. Consequently, \textit{recent} input-output data of length~$T_\mathrm{ini}=\bar{n}=6\geq\mathbf{l}(\mathcal{G})$ are collected with inputs from  the uniform distribution in $[-1, 1]$. Moreover, recent data is corrupted by input disturbances and output noises as in Section~\ref{subsec.FinalControlDesignProblem}, where the trajectories $w$ and $\bar{d}$ are selected to satisfy quadratic constraints \eqref{eq.yini-noiseconstraint} and \eqref{eq.distconstraint}, respectively, with  
\begin{equation}\label{eq.simNoisePara}
\begin{split}
\Phi_{11} = T_\mathrm{ini}p\epsilon,\enskip &\Phi_{12}=\0,\enskip \Phi_{22} = -\I, \\
\Phi_{d,11} = \left(T_\mathrm{ini}+T_f\right)m&\epsilon,\enskip \Phi_{d,12}=\0,\enskip \Phi_{d,22} = -\I, 
\end{split}
\end{equation}
and $\epsilon=0.001$.
We are interested in 
solving the min-max problem~\eqref{eq.overall_minimax}.
We select $T_f=20$, $r=\0$, $Q=\I$, and $R=\I$ to robustly regulate the output of the system to zero within a horizon of length $20$. Moreover, we seek to do so while ensuring that $\check{u}$ and $y$ satisfy quadratic constraints~\eqref{eq.I_cons_quad},~\eqref{eq.O_cons_quad} with
\begin{equation*}
\begin{split}
\Psi_{11} = T_f\epsilon_u,\enskip &\Psi_{12}=\0,\enskip \Psi_{22} = -\I, \\
\Theta_{11} = T_f\epsilon_y,\enskip &\Theta_{12}=\0,\enskip \Theta_{22} = -\I,
\end{split}
\end{equation*}
and $\epsilon_u=\epsilon_y=0.5$.

As shown in Theorem~\ref{thm.LMIopt_overall}, it is possible to convert this robust control input design problem into the minimization problem~\eqref{eq.finalOpt_overall}. This problem is then solved using Yalmip~\cite{Lofberg2004} on Matlab with MOSEK~\cite{mosek} specified as the solver, which returns the optimal control sequence $u$. This control sequence is tested with multiple compatible realizations of noise trajectories. In particular, we randomly select $100$ vectors $\bar{g}$ that satisfy the quadratic constraints~\eqref{eq.gbar_const} and \eqref{eq.gbar_const2}, which, in view of~\eqref{eq.rho_parameterization}, parameterize $100$ feasible sequences of $w$ and $\bar{d}$.
As shown in Figure~\ref{fig.LTI_outputs}, output trajectories are quickly brought around zero for all noise and disturbance realizations. 
Moreover, Figure~\ref{fig.LTI_constraints} displays the robustness of the closed-loop system. Specifically, the first plot in Figure~\ref{fig.LTI_constraints} shows that $\gamma\leq \gamma^*$ for all selected noise and disturbance realizations, where each blue circle corresponds to a specific realization and  $\gamma^*$ is the optimal value of~\eqref{eq.finalOpt_overall}. Moreover, the second and third plots show that the input and output constraints are satisfied for all selected noise and disturbance trajectories, i.e., $\psi(\check{u})\geq 0$ and $\theta(y)\geq 0$, respectively. 

\begin{figure}[t]
	\includegraphics[width=0.45\textwidth]{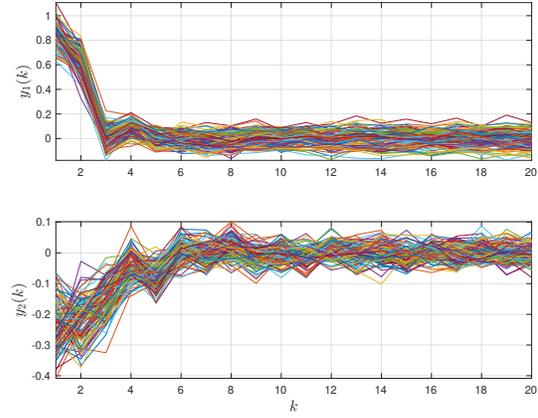}
	\caption{Output responses with the designed robust control under different feasible noise and disturbance trajectories.} \label{fig.LTI_outputs}
\end{figure}

\begin{figure}[t]
	\includegraphics[width=0.45\textwidth]{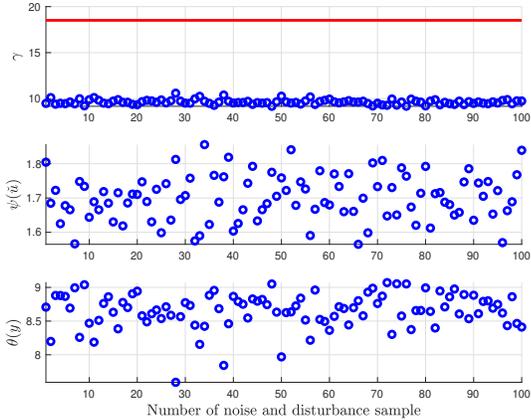}
	\caption{Robustness against different noise and disturbance trajectories. Top: Tracking errors $\gamma=y^\top Q y+\check{u}^\top R \check{u}$ (blue circles), and the optimal value $\gamma^*$ (red line) computed from~\eqref{eq.finalOpt_overall}. Middle: Values of the input constraint $\psi(\check{u})$ computed as in~\eqref{eq.I_cons_quad}. Bottom: Values of the output constraint $\theta(y)$ computed as in~\eqref{eq.O_cons_quad}.} \label{fig.LTI_constraints}
      \end{figure}
      
When parameterizing the noise trajectories as in~\eqref{eq.rho_parameterization}, the parameterization methods proposed in Sections~\ref{subsec.ReducingDimension} and \ref{subsec.InputDist} allow for a significant reduction in the sizes of the LMI conditions~\eqref{eq.finalLMI_dist}, \eqref{eq.I_cons_dist_LMI}, and \eqref{eq.O_cons_dist_LMI}. In particular, the size of each LMI condition is reduced from $291$ to $143$.

It is seen from Figure~\ref{fig.LTI_constraints} that the obtained $\gamma$ values are not as high as the optimal value $\gamma^*$. Moreover, the input and output constraints are not active in any of the different simulation scenarios, i.e., $\psi(\check{u})\neq 0$ and $\theta(y)\neq 0$. These limitations are due to the fact that the version of S-lemma for multiple quadratic inequalities
and the semidefinite version of Schur complement used in Section~\ref{subsec.FinalControlDesignProblem} are conservative. In order to demonstrate that they are the only sources of conservativity, we run another simulation with the same LTI system, in which we do not consider actuator disturbances and input/output constraints.
Specifically, the same \textit{historical} data as in the previous simulation are used to construct the Hankel matrices $U_p$, $U_f$, $Y_p$, and $Y_f$. Moreover, the same input sequence $u_\mathrm{ini}$ is used to generate the \textit{recent} trajectory. Differently to the previous case, only the recent output trajectory $y_\mathrm{ini}$ is affected by measurement noise $w$. This noise is chosen to satisfy~\eqref{eq.yini-noiseconstraint} with the  matrix $\Phi$ defined by~\eqref{eq.simNoisePara}. The $T_f$, $r$, $Q$, and $R$ of the previous example are chosen to ensure robust regulation of system output to zero. By utilizing the results of Section~\ref{sec.RobustControlDesign}, we solve the problem~\eqref{eq.finalOpt} with the matrix $M$ defined in~\eqref{eq.M_selection}. Similarly to the previous simulation, the calculated control input $u$ is used to control the system with $100$ different realizations of the vector $g_w$ parameterizing different feasible noise trajectories $w$. The results of this simulation are presented in Figure~\ref{fig.LTI_constraints_2}, where one sees that the tracking costs $\gamma$ in blue circles are smaller than the robust optimal tracking cost $\gamma^*$ computed from~\eqref{eq.finalOpt} for all feasible noise trajectories. It can also be seen from this figure that some $\gamma$ values get quite close to $\gamma^*$, hence supporting the claim that Theorem~\ref{thm.LMIopt} in Section~\ref{sec.RobustControlDesign} is not conservative.

\begin{figure}[t]
	\includegraphics[width=0.45\textwidth]{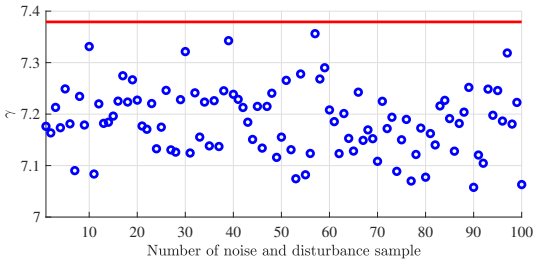}
	\caption{Tracking errors $\gamma=y^\top Q y+u^\top R u$ (blue circles), and the optimal value $\gamma^*$ (red line) computed from~\eqref{eq.finalOpt}.} \label{fig.LTI_constraints_2}
\end{figure}


\section{Conclusions}\label{sec.Conclusion}
Willems' Fundamental Lemma shows that finite-length persistently exciting data can characterize the behaviors of linear systems, which enables data-driven simulation and control.
In this paper, we build on this data-dependent representation to consider the case that the recent output data are noisy and solve worst-case robust optimal tracking control problems in a data-driven fashion.
The key ingredient of our approach is a suitable parameterization of the feasible noise trajectories and the performance specification, which allows one to express them as quadratic constraints.
Then, by applying the S-lemma, we show that the worst-case robust control problem is equivalent to a minimization problem with a linear cost and LMI constraints.
Moreover, by carefully selecting the noise parameterization, we can show that the dimension of the LMI optimization problem does not scale with the length of historical data.
Our method can also easily incorporate input and output constraints, as well as actuator disturbances.

At present, the proposed method assumes that noise affects only recent data.
Future work will be devoted to generalizations accounting for noise also in historical data.

\bibliographystyle{ieeetr}
\bibliography{references}
\end{document}